\documentclass[12pt]{amsart}

\usepackage{amssymb, mathrsfs}
\usepackage{graphicx}

\usepackage{lineno}

\usepackage{color}

\newtheorem{theorem}{Theorem}

\newtheorem{observation}[theorem]{Observation}
\newtheorem{corollary}[theorem]{Corollary}

\newtheorem{problem}{Problem}

\newtheorem{them}{Theorem}
\newtheorem{lema}[them]{Lemma}

\newtheorem{example}[theorem]{Example}
\newtheorem{remark}[theorem]{Remark}

\begin{document}

\title[Path addition]{Changing and unchanging of the domination number of a graph: path addition numbers }

\author[]{Vladimir Samodivkin}
\address{Department of Mathematics, UACEG, Sofia, Bulgaria}
\email{vl.samodivkin@gmail.com}
\today

\begin{abstract} 
Given a graph $G = (V,E)$ and two its distinct vertices $u$ and $v$. 
The $(u,v)$-$P_k$-{\em addition graph} of $G$ is the graph 
$G_{u,v,k-2}$ obtained from disjoint union of $G$ and a path 
$P_k: x_0,x_1,..,x_{k-1}$, $k \geq 2$, by identifying 
the vertices $u$ and $x_0$, and identifying the vertices $v$ and $x_{k-1}$.  
We prove that 
(a) $ \gamma(G)-1 \leq \gamma(G_{u,v,k})$ for all $k \geq 1$, and 
(b)  $\gamma(G_{u,v,k})  > \gamma(G)$ when $k \geq 5$. 
We also provide  necessary and sufficient conditions for 
 the equality  $\gamma(G_{u,v,k})  =  \gamma(G)$ to be valid  for each  
pair $u,v \in V(G)$. 
\end{abstract}

\maketitle


\begin{flushleft}
\end{flushleft}


\section{Introduction}
For basic notation and graph theory terminology not explicitly defined here, we
in general follow Haynes et al. \cite{hhs1}.
We denote the vertex set and the edge set of a graph $G$ by $V(G)$ and $ E(G),$  respectively. 
The complement $\overline{G}$ of $G$ is the graph whose
vertex set is $V (G)$ and whose edges are the pairs of nonadjacent vertices of $G$. 
We write (a) $K_n$ for the {\em complete graph} of order $n$, 
(b) $K_{m,n}$ for the {\em complete bipartite graph} with partite sets of order $m$ and $n$, and 
(c) $P_n$ for the  {\em path} on $n$ vertrices. 
Let $C_m$ denote the {\em cycle} of length $m$. 
	For any vertex $x$ of a graph $G$,  $N_G(x)$ denotes the set of all  neighbors of $x$ in $G$,  
	$N_G[x] = N_G(x) \cup \{x\}$ and the degree of $x$ is $deg(x,G) = |N_G(x)|$. 
The {\em minimum} and {\em maximum} degrees
 of a graph $G$ are denoted by $\delta(G)$ and $\Delta(G)$, respectively.
For a subset $A \subseteq V (G)$, let $N_G(A) = \cup_{x \in A} N_G(x)$  and $N_G[A] = N_G(A) \cup A$.  
A {\em vertex cover} of a graph is a set of vertices such that 
each edge of the graph is incident to at least one vertex of the set.
Let $G$ be a graph and $uv$ be an edge of $G$. By subdividing
the edge $uv$ we mean forming a graph $H$ from $G$ by adding a new vertex $w$
and replacing the edge $uv$ by $uw$ and $wv$. 
Formally, $V (H) = V (G) \cup \{w\}$ and $E(H) = (E(G) -\{uv\}) \cup \{uw, wv\}$. 

 The study of domination and related subset problems is one of the fastest growing areas in graph theory.
	For a comprehensive introduction to the theory of domination in graphs we refer the reader
 to Haynes et al. \cite{hhs1}.
A {\em dominating set} for a graph $G$ is a subset $D\subseteq V(G)$ of 
vertices such that every vertex not in $D$ is adjacent to at least one vertex in $D$. 
The  {\em domination number} of $G$, denoted by $\gamma (G)$,
	is the smallest cardinality of a dominating set of $G$.
   A dominating set of $G$ with cardinality $\gamma(G)$ is called a 
   $\gamma$-{\em set of} $G$.
The concept of $\gamma$-bad/good vertices in graphs was introduced by 
Fricke et al. in \cite{fhhhl}.  A vertex $v$ of a graph $G$ is called:
\begin{itemize}
\item[(i)] \cite{fhhhl} $\gamma$-{\em good}, if $v$ belongs to some $\gamma$-set of $G$ and
\item[(ii)] \cite{fhhhl} $\gamma$-{\em bad}, if $v$ belongs to no $\gamma$-set of $G$.
\end{itemize}
A graph $G$ is said to be $\gamma$-{\em excellent} whenever all its vertices are $\gamma$-good \cite{fhhhl}. 
Brigham et al. \cite{bcd} defined 
(a) a vertex $v$ of a graph $G$ to be  {\em $\gamma$-critical} if $\gamma(G-v) < \gamma(G)$, 
 and $G$ to be  {\em vertex domination-critical} 
(from now  on called {\em vc-graph}) 
if each  vertex of $G$ is $\gamma$-critical. 
For a graph G we define:
$V^-(G) = \{x \in V(G) \mid \gamma(G-x) < \gamma(G)\}$.

It is often of interest to known how the value of a graph parameter $\mu$ is affected when a
 change is made in a graph, for instance vertex or edge removal, edge addition, 
edge subdivision  and edge contraction. 
In this connection, here we consider this question in the case
$\mu = \gamma$ when a path is added to a graph.

 Path-addition is an operation that takes a graph and adds an
internally vertex-disjoint path between two vertices together with a set
of supplementary edges. 
This operation can be considered as a natural generalization of the edge addition. 
Formally,  let $u$ and $v$ be distinct vertices of a graph $G$. 
The $(u,v)$-$P_k$-{\em addition graph} of $G$ is the graph 
$G_{u,v,{k-2}}$ obtained from disjoint union of $G$ and a path 
$P_k: x_0,x_1,..,x_{k-1}$, $k \geq 2$, by identifying 
the vertices $u$ and $x_0$, and identifying the vertices $v$ and $x_{k-1}$.   
When $k \geq 3$ we call $x_1, x_2,..,x_{k-2}$ {\em path-addition vertices}. 
By $pa_{\gamma}(u,v)$ we denote the minimum number $k$ such that 
$\gamma(G) < \gamma(G_{u,v,k})$. 
For every graph $G$ with at least $2$ vertices we define

\begin{itemize}
\item[$\vartriangleright$]  the {\em $e$-path addition ($\overline{e}$-path addition) 
number with respect to domination}, denoted $epa_{\gamma}(G)$ 
($\overline{e}pa_{\gamma}(G)$, respectively), to be 
\begin{itemize}
\item[$\bullet$] $epa_{\gamma}(G) = \min\{pa_{\gamma}(u,v) \mid u,v \in V(G), uv \in E(G)\}$, 
\item[$\bullet$]  $\overline{e}pa_{\gamma}(G) = \min\{pa_{\gamma}(u,v) \mid u,v \in V(G), uv \not \in E(G)\}$, and
\end{itemize}
\item[$\vartriangleright$]   the {\em upper $e$-path addition 
(upper $\overline{e}$-path addition) 
number with respect to domination}, denoted $Epa_{\gamma}(G)$ 
($\overline{E}pa_{\gamma}(G)$, respectively), to be 
\begin{itemize}
\item[$\bullet$] $Epa_{\gamma}(G) = \max\{pa_{\gamma}(u,v) \mid u,v \in V(G), uv \in E(G)\}$, 
\item[$\bullet$]  $\overline{E}pa_{\gamma}(G) = \max\{pa_{\gamma}(u,v) \mid u,v \in V(G), uv \not \in E(G)\}$.
\end{itemize}
\end{itemize}
If $G$ is complete then we write 
$\overline{E}pa_{\gamma} (G) = \overline{e}pa_{\gamma}(G) = \infty$, 
and if $G$ is edgeless then  $epa_{\gamma}(G) = Epa_{\gamma}(G) = \infty$. 
In what follows the subscript $\gamma$ will be omitted from the notation.

The remainder of this paper is organized as follows.
In Section 2:  (a) 	we prove that $1 \leq epa(G) \leq 3$ and $2\leq Epa(G) \leq 3$,  and 
                          (b)  we present necessary and sufficient conditions  for $pa(u,v) =i$, $i=1,2,3$, 
													        where $uv \in E(G)$. 
In Section 3: (c) we show that $1 \leq \overline{e}pa(G)  \leq \overline{E}pa(G) \leq 5$, and 
                         (d) we give necessary and sufficient conditions  for 
												       $\overline{e}pa(G)  = \overline{E}pa(G) = j$, $1 \leq j \leq 5$.
We conclude in Section 4 with  open problems.

We end this section with some known results which will be useful in proving
our main results.

\begin{lema}\label{subdiv} \cite{bv}
If $G$ is a graph and $H$ is any graph obtained from $G$ by
subdividing some edges of $G$, then $\gamma(H)  \geq \gamma(G)$.
\end{lema}

\begin{lema}\label{folk}
 Let $G$ be a graph and $v \in V(G)$. 
\begin{itemize}
\item[(i)]  \cite{fhhhl} If $v$ is $\gamma$-bad, then $\gamma(G-v) = \gamma(G)$. 
\item[(ii)]  \cite{bcd}  $v$ is $\gamma$-critical if and only if $\gamma(G-v) = \gamma(G)-1$. 
\item[(iii)] \cite{fhhhl} If $v$ is $\gamma$-critical, then all its neighbors are $\gamma$-bad vertices of $G-v$.
\item[(iv)] \cite{wa} If $e \in E(\overline{G})$ then $\gamma(G) - 1 \leq \gamma(G+e) \leq \gamma(G)$. 
\end{itemize}
\end{lema}

In most cases, Lemma \ref{folk} will be used in the sequel without specific
reference.

\section{The adjacent case}

The aim of this section is to prove that  $1 \leq pa(u,v) \leq 3$  and 
 to find necessary and sufficient conditions for $pa(u,v) =i$, $i=1,2,3$, 
where $uv \in E(G)$.

\begin{observation}\label{chain1}
If $u$ and $v$ are adjacent vertices of a graph $G$,  
 then $\gamma(G)  = \gamma (G_{u,v,0}) 
                                      \leq \gamma (G_{u,v,k}) \leq \gamma (G_{u,v,k+1})$ for $k \geq 1$. 
\end{observation}
\begin{proof} 
The equality $\gamma(G)  = \gamma (G_{u,v,0}) $ is obvious. 
 For any $\gamma$-set $M$ of  $G_{u,v,1}$ both 
$M_u = (M-\{x_1\})\cup\{u\}$ and  $M_v = (M-\{x_1\})\cup\{v\}$ are 
dominating sets of $G$, and at least one of them is a $\gamma$-set of $G_{u,v,1}$.
  Hence $\gamma(G) \leq \min \{|M_u|, |M_v|\} = \gamma(G_{u,v,1})$. 
The rest follows by Lemma \ref{subdiv}. 
\end{proof}

\begin{theorem} \label{pa1}
Let $u$ and $v$ be adjacent vertices of a graph $G$. 
 Then $\gamma(G) \leq \gamma (G_{u,v,1}) \leq \gamma(G) +1$ 
and  the following is true: 
\begin{itemize}
\item[(i)]  $\gamma(G) =\gamma (G_{u,v,1})$ if and only if
 at least one of  $u$ and $v$ is a $\gamma$-good vertex of $G$.
\item[(ii)]  $\gamma (G_{u,v,1}) =  \gamma(G) + 1$  if and only if
 both $u$ and $v$ are $\gamma$-bad vertices of $G$. 
\end{itemize}
 \end{theorem}
 
\begin{proof}
The left side inequality follows by Observation \ref{chain1}. 
If $D$ is a $\gamma$-set of $G$ then $D \cup \{x_1\}$ is a dominating set of 
$G_{u,v,1}$, which implies $\gamma (G_{u,v,1}) \leq \gamma(G) +1$. 

If at least one of $u$ and $v$ belongs to  some $\gamma$-set $D_1$ of $G$, 
then $D_1$ is a dominating set of $G_{u,v,1}$. This clearly 
 implies  $\gamma(G) = \gamma(G_{u,v,1})$.

Let now both $u$ and $v$ are $\gamma$-bad vertices of $G$, 
and  suppose that $\gamma (G_{u,v,1}) = \gamma(G)$. 
In this case  for any $\gamma$-set $M$ of  $G_{u,v,1}$ is fulfilled 
 $u,v \not\in M$ and $x_1 \in M$. But then $(M-\{x_1\}) \cup \{u\}$
 is a $\gamma$-set for both $G$ and $G_{u,v,1}$, a contradiction.
\end{proof}

\begin{corollary} \label{1}
Let $G$ be a graph with edges. Then (a) $Epa(G) \geq  2$, 
and (b) $epa(G) = 1$ if and only if 
the set of all $\gamma$-bad vertices of $G$ is neither empty nor independent.
\end{corollary}

\begin{theorem}\label{epn=2}
Let $u$ and $v$ be adjacent vertices of a graph $G$. 
Then  $\gamma(G) \leq \gamma (G_{u,v,2}) \leq \gamma(G) +1$. 
Moreover, 
\begin{itemize}
\item[($\mathbb{A}$)] $\gamma (G_{u,v,2}) = \gamma(G) +1$ if and only if 
                                              one of the following holds:
\begin{itemize}
\item[(i)] both $u$ and $v$ are $\gamma$-bad vertices of $G$,
\item[(ii)] at least one of $u$ and $v$ is $\gamma$-good, 
            $u,v \not\in V^-(G)$ and each $\gamma$-set of $G$ contains 
            at most one of $u$ and $v$. 
\end{itemize}
\item[($\mathbb{B}$)] $\gamma (G_{u,v,2}) = \gamma(G) $ if and only if 
                                              at least one of the following is true:
\begin{itemize}
\item[(iii)] there exists a $\gamma$-set of $G$ which contains both $u$ and $v$, 
\item[(iv)] at least one of $u$ and $v$ is in  $V^-(G)$.  
\end{itemize}
\end{itemize}
\end{theorem}

\begin{proof}
The left side inequality follows by Observation \ref{chain1}. 
If $D$ is an arbitrary $\gamma$-set of $G$,  
then $D\cup \{x_1\}$ is a dominating set of $G_{u,v,2}$. 
Hence $\gamma(G_{u,v,2}) \leq \gamma(G) + 1$. 

($\mathbb{A}$) $\Rightarrow$  Assume that the equality  $\gamma (G_{u,v,2}) = \gamma(G) +1$ holds. 
By Theorem \ref{pa1} we know that $\gamma(G_{u,v,1}) \in \{\gamma(G), \gamma(G) + 1\}$. 
If $\gamma(G_{u,v,1}) = \gamma(G) + 1$ then 
again by Theorem \ref{pa1}, 
both $u$ and $v$ are $\gamma$-bad vertices of $G$. 
So let $\gamma(G) = \gamma (G_{u,v,1})$. 
Then at least one of $u$ and $v$ is a $\gamma$-good vertex of $G$ (Theorem \ref{pa1}). 
Clearly there is no $\gamma$-set of $G$ which contains both $u$ and $v$. 
If $u \in V^-(G)$ and $U$ is a $\gamma$-set of $G-u$, then 
$U \cup \{x_1\}$ is a dominating set of $G_{u,v,2}$ and $|U \cup \{x_1\}| = \gamma(G)$, 
a contradiction. Thus $u,v \not\in V^-(G)$.  

($\mathbb{A}$) $\Leftarrow$ If both $u$ and $v$ are $\gamma$-bad vertices of $G$,
             then $\gamma(G_{u,v,1}) = \gamma(G) + 1$(Theorem \ref{pa1}).  
             But we know that $\gamma(G_{u,v,1}) \leq \gamma(G_{u,v,2}) \leq \gamma(G) + 1$;
             hence  $\gamma(G_{u,v,2})= \gamma(G) + 1$. 
 Finally let  (ii) hold and $M$ a $\gamma$-set of $G_{u,v,2}$.  
 If $x_1,x_2 \not\in M$ then $u,v \in M$ which leads to $\gamma(G_{u,v,2}) > \gamma(G)$. 
If $x_1,x_2 \in M$ then $(M-\{x_1, x_2\}) \cup \{u,v\}$ is a dominating set of $G$
of cardinality more than $\gamma(G)$. 
Now let without loss of generality $x_1 \in M$ and $x_2 \not\in M$. 
If $M-\{x_1\}$ is a dominating set of $G$, then 
$\gamma(G) + 1 \leq |M| = \gamma(G_{u,v,2}) \leq \gamma(G) + 1$. 
So, let $M-\{x_1\}$ be no dominating set of $G$. 
Hence $M-\{x_1\}$ is a dominating set of $G-u$. 
Since $u \not\in V^-(G)$,
 $\gamma(G) \leq \gamma(G-u) \leq |M-\{x_1\}| <  \gamma(G_{u,v,2})$. 

($\mathbb{B}$) Immediately by ($\mathbb{A}$) and 
$\gamma(G) \leq \gamma (G_{u,v,2}) \leq \gamma(G) +1$. 
\end{proof}

The {\em independent domination number} of a graph $G$, 
denoted by $i(G)$, is the minimum size of an independent dominating set of $G$. 
It is obviously that $i(G) \geq \gamma(G)$. 
In a graph $G$, $i(G)$ is {\em strongly equal to} $\gamma(G)$, written $i(G) \equiv \gamma(G)$,  
if each $\gamma$-set of $G$ is independent. 
It remains an open problem to characterize the graphs $G$  with $i(G) \equiv \gamma(G)$ \cite{gh}. 

\begin{corollary}\label{22}
Let $G$ be a graph with edges. 
Then (a) $epa(G) \geq  2$ if and only if 
the set of all $\gamma$-bad vertices is  either empty or independent, 
and (b) $Epa(G)=2$  if and only if $i(G) \equiv \gamma(G)$. 
\end{corollary}

\begin{proof}
(a) Immediately by Corollary \ref{1}.

(b) $\Rightarrow$  Let $Epa(G)=2$. If $D$ is a $\gamma$-set of $G$ and 
                                     $u,v \in D$ are adjacent, then $D$ is a dominating set of 
																		$G_{u,v,2}$, a contradiction. 

(b) $\Leftarrow$      Let all $\gamma$-sets of $G$ be independent. 
                                     Suppose $u \in V^-(G)$ and $D$ a $\gamma$-set of $G-u$. 
																		Then $D_1 = D \cup \{v\}$ is a $\gamma$-set of $G$, 
																		where $v$ is any neighbor of $u$. But $D_1$ is not independent. 
																		Hence $V^-(G)$ is empty. 
																		Thus, for any $2$ adjacent vertices $u$ and $v$ of $G$ is fulfilled 
																		either ($\mathbb{A}$)(i) or ($\mathbb{A}$)(ii) of Theorem \ref{epn=2}. 
																		Therefore $Epa(G) \leq 2$.  The result now follows by Corollary \ref{1}. 
\end{proof}

Denote by $\mathbb{Z}_n = \{0, 1, . . . , n -1\}$  the additive  group of order $n$. 
Let $S$ be a subset of $\mathbb{Z}_n$ such that $0\not\in S$ and $x \in S$ implies $-x \in S$. 
The {\em circulant graph} with distance set $S$ is the graph $C(n; S)$ with vertex set $\mathbb{Z}_n$
 and vertex $x$ adjacent to vertex $y$ if and only if $x - y \in S$.

Let $n\geq 3$ and $k \in \mathbb{Z}_n  - \{0\}$. 
The {\em generalized Petersen graph} $P(n, k)$ is the graph on the vertex-set
$\{x_i , y_i \mid  i \in \mathbb{Z}_n \}$ with adjacencies $x_ix_{i+1}, x_iy_i$,
and $y_i y_{i+k}$ for all $i$.

\begin{example}\label{effdom}
A special case of graphs $G$ with $Epa(T) =2$ 
are graphs for which each $\gamma$-set is efficient dominating 
(an efficient dominating set  in a graph $G$ is a set $S$ such that 
$\{N[s] \mid s \in S\}$ is a partition of $V(G)$). 
We list several examples of such graphs \cite{samtc}:
 \begin{itemize}
\item[(a)]  A {\em crown graph} $H_{n,n}$, $n \geq 3$, which is obtained from 
       the complete bipartite graph $K_{n,n}$ by removing a perfect matching. 
\item[(b)]  Circulant graphs $G = C(n=(2k+1)t; \{1,..,k\} \cup \{n-1,...,n-k\} )$, where  $k,t \geq 1$. 
      
\item[(c)]  Circulant graphs $G = C(n; \{\pm 1, \pm s\} )$, where $2 \leq s \leq n-2$,   $s \not=  n/2$, 
									$5 | n$ and $s \equiv \pm2 \pmod 5$.
			
\item[(d)] The {\em generalized Petersen graph} $P(n, k)$, where		$n \equiv 0 \pmod 4$ and $k$ is odd.
\end{itemize}							
\end{example}

\begin{theorem}\label{tri} 
If $u$ and $v$ are adjacent vertices of a graph $G$, then 
$\gamma(G_{u,v,3}) = \gamma(G) + 1$. 
\end{theorem}
\begin{proof}
If $D$ is a $\gamma$-set of $G$, then $D \cup \{x_2\}$ is a dominating set of $G$. 
Hence $\gamma(G_{u,v,3}) \leq \gamma(G) + 1$. 

Let $M$ be a $\gamma$-set of $G_{u,v,3}$.  Then at least one of 
$x_1, x_2$ and $x_3$ is in $M$. If $x_2 \in M$ then clearly $\gamma(G_{u,v,3}) = \gamma(G) + 1$. 
If $x_2 \not\in M$ and $x_1,x_3 \in M$,  then $(M-\{x_1,x_3\}) \cup \{u\}$ is a dominating 
set of $G$.  If $x_2,x_3 \not\in M$ and $x_1 \in M$, then $v \in M$ and 
$M-\{x_1\}$ is a dominating set of $G$. All this leads to $\gamma(G_{u,v,3}) = \gamma(G) + 1$. 
 \end{proof}

\begin{corollary}\label{epn3}
Let $G$ be a graph with edges. Then  $epa(G) \leq Epa(G) \leq 3$. 
Moreover, 
(a) $Epa(G) =3$ if and only if $G$ has a $\gamma$-set that is not independent,
      and 
(b) $epa(G) = 3$ 	if and only if for each  pair of adjacent vertices $u$ and $v$ 
          at least one of (iii) and (iv) of Theorem \ref{epn=2}  ($\mathbb{B}$) is valid. 		
\end{corollary}
\begin{proof}
By Theorem \ref{tri} , $epa(G) \leq Epa(G) \leq 3$. 

(a) Immediately by Corollary \ref{22} and Theorem \ref{tri}.

(b) Theorem \ref{epn=2}, Corollary \ref{22} and Theorem \ref{tri}  
       together immediately imply the required. 
\end{proof}

\begin{corollary}\label{epn3v}
Let $G$ be a graph with edges.  
If $V^-(G)$ has a subset which is a vertex cover of $G$, 
then $epa(G) = 3$. In particular, if $G$ is a vc-graph then $epa(G) = 3$.
\end{corollary}

We need to define the following classes of graphs $G$ with $\Delta(G) \geq 1$: 
\begin{itemize}
\item[$\bullet$] $\mathcal{A} = \{G \mid epa(G) = 3\}$,
\item[$\bullet$]  $\mathcal{A}_1 = \{G \mid V^-(G) \mbox{\ is a vertex cover of \ } G\}$, 
\item[$\bullet$]  $\mathcal{A}_2 = \{G \mid  \mbox{\  each two adjacent verices belongs to some \ } 
                                     \gamma\mbox{-set of \ } G\}$, 
\item[$\bullet$]  $\mathcal{A}_3 = \{G \mid  G \mbox{\ is a vc-graph}\}$.
   \end{itemize}

Clearly, $\mathcal{A}_3 \subseteq \mathcal{A}_1$ and by 
Corolaries \ref{epn3} and \ref{epn3v}, $\mathcal{A}_1 \cup \mathcal{A}_2 \subseteq \mathcal{A}$.
These relationships are illustrated in the Venn diagram of Fig. \ref{fig:opit1}(left). 
To continue we need to relabel this diagram in six regions $R_0-R_5$ as shown in Fig. \ref{fig:opit1}(right). 
In what follows in this section we show that none of $R_0-R_5$ is empty.
The {\em corona}  of a graph $H$ is the graph $G=H\circ K_1$ obtained from
$H$  by adding a degree-one neighbor to every vertex of $H$.

\begin{figure}[htbp]
	\centering
		\includegraphics{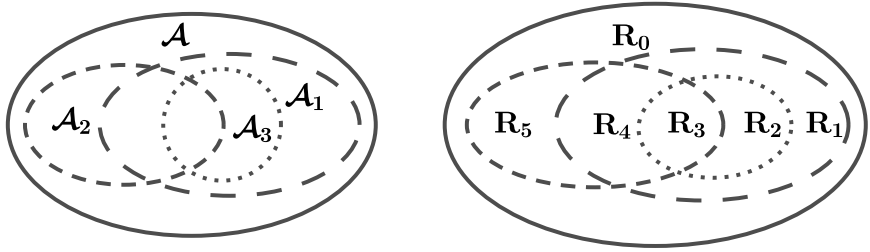}
	\caption{{\em Left}:  Classes of graphs with $epa = 3$. {\it Right}: Regions of Venn diagram.}
	\label{fig:opit1}
\end{figure}

\begin{remark}\label{-0}
It is easy to see that all the following hold: 
\begin{itemize}
\item[(i)]  If $H$ is a connected graph of  order $n \geq 2$, then $G=H\circ K_1 \in \mathbf{R_0}$. 
\item[(ii)] Let $G$ be a graph obtained by $C_7: x_0,x_1,..,x_6,x_0$ by adding 
                    a vertex $y$ and edges $yx_0, yx_2$. Then $G$ is in $\mathbf{R_1}$.
\item[(iii)]  The graph $G_{10}$  depicted in  Fig. \ref{fig:vc3} is in 
                       $\mathcal{A}_3$  and $\gamma(G_{10}) = 3$ \cite{aaa}. 
												It is obvious that no $\gamma$-set of $G_{10}$ contains both $u$ and $v$. 
												Hence $G_{10} \in \mathbf{R_2}$.                       
\item[(iv)] $C_{3k+1} \in \mathbf{R_3}$ for all $k \geq 1$.
\item[(v)]  $K_{2,n} \in \mathbf{R_4}$ for all $n \geq 3$.
\item[(v)]  $K_{n,n} \in \mathbf{R_5}$ for all $n \geq 3$.
\end{itemize}
Thus all regions $\mathbf{R_0}, \mathbf{R_1}, \mathbf{R_2}, \mathbf{R_3}, \mathbf{R_4}, \mathbf{R_5}$ 
                                are nonempty.
\end{remark}

\begin{figure}[htbp]
	\centering
		\includegraphics{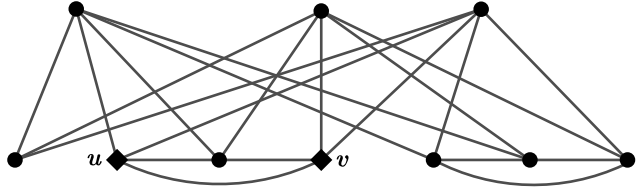}
	\caption{Graph $G_{10}$ is in $\mathbf{R_2}$}
	\label{fig:vc3}
\end{figure}

\newpage

\section{The nonadjacent case}  
In this section we show that  $1 \leq \overline{e}pa(G) \leq \overline{E}pa(G) \leq 5$ and 
we obtain necessary and sufficient conditions  for 
												       $\overline{e}pa(G)  = \overline{E}pa(G) = j$, $1 \leq j \leq 5$.

We begin with an easy observation which is an immediate consequence by 
Lemma \ref{folk}(iv) and Lemma \ref{subdiv}.

\begin{observation}\label{chain2}
Let $u$ and $v$ be nonadjacent vertices of a graph $G$. 
 Then $\gamma(G) -1 \leq \gamma (G_{u,v,0}) \leq \gamma(G)$ and $\gamma (G_{u,v,k}) \leq \gamma (G_{u,v,k+1})$ for $k \geq 0$. 
\end{observation}

 \begin{theorem} \label{pa1over}
 Let $u$ and $v$ be nonadjacent vertices of a graph $G$. 
Then $\gamma(G) -1 \leq \gamma (G_{u,v,1}) \leq \gamma(G) +1$. 
  Moreover,
\begin{itemize}
\item[(i)]  $\gamma(G) - 1 =\gamma (G_{u,v,1})$ if and only if 
                        $\gamma(G-\{u,v\}) = \gamma(G) - 2$.  
\item[(ii)]			                          $\gamma (G_{u,v,1}) = \gamma(G) +1$ 
                         	if and only if 	both $u$ and $v$ are $\gamma$-bad vertices of $G$, 	
													$u \not\in V^-(G-v)$ and $v \not\in V^-(G-u)$.
	                       If $\gamma (G_{u,v,1}) = \gamma(G) +1$ then $x_1 \in  V^-(G_{u,v,1})$. 
\end{itemize}
\end{theorem}

\begin{proof}
Let $M$ be any $\gamma$-set of  $G_{u,v,1}$. 
Then at least one and not more than two of $x_1,u$ and $v$ 
must be in $M$. Hence $M_1 = (M-\{x_1\}) \cup \{u,v\}$ is 
a dominating set of $G$ and $|M_1| \leq |M|+1$. 
This implies $\gamma(G) \leq \gamma(G_{u,v,1}) +1$. 

(i) $\Rightarrow$ 
Assume the equality $\gamma(G) -1 = \gamma(G_{u,v,1})$ holds. 
Then $|M_1| = |M| + 1$  and $M_1$ is a $\gamma$-set of $G$. 
Hence $x_1 \in M$ and $pn[x_1, M] = \{x_1,u,v\}$. 
Since $M_1 - \{u,v\}$ is a dominating set of $G-\{u,v\}$, we have 
$\gamma(G)-2 \leq \gamma(G-\{u,v\}) \leq |M_1-\{u,v\}| = \gamma(G)-2$. 

(i) $\Leftarrow$
Suppose now $\gamma(G-\{u,v\}) = \gamma(G)-2$. 
Then for any $\gamma$-set $U$ of $G-\{u,v\}$, the set $U \cup \{x_1\}$ is a dominating set of 
$G_{u,v,1}$. This leads to  $\gamma(G_{u,v,1}) \leq |U \cup \{x_1\}| = \gamma(G)-1 \leq \gamma(G_{u,v,1})$. 
\medskip

Now we will prove the right side inequality. 
Let $D$ be any $\gamma$-set of $G$. 
If at least one of $u$ and $v$ is in $D$, then $D$ is a dominating set $G_{u,v,1}$ 
and $\gamma(G_{u,v,1}) \leq \gamma(G)$.  So, let 
neither $u$ nor $v$ belong to some $\gamma$-set of $G$. 
Then $D \cup \{x_1\}$ is a dominating set of $G_{u,v,1}$ and 
$\gamma(G_{u,v,1}) \leq    \gamma(G) + 1$. 

(ii) $\Rightarrow$
Assume  that $\gamma(G_{u,v,1}) = \gamma(G) + 1$. 
Then $u$ and $v$ are $\gamma$-bad vertices of $G$ and for any $\gamma$-set 
$D$ of $G$, $D \cup \{x_1\}$ is a $\gamma$-set of $G_{u,v,1}$. 
Hence  $x_1 \in  V^-(G_{u,v,1})$. 
Suppose $u \in V^-(G-v)$ and let $U$ be a $\gamma$-set of $G-\{u,v\}$. 
Then $U_1 = U \cup \{x_1\}$ is a dominating set of $G_{u,v,1}$ and 
 $\gamma(G) + 1 = \gamma(G_{u,v,1}) \leq |U_1| = 1 + \gamma ((G-v)-u) =\gamma(G-v) = \gamma (G)$, 
a contradiction.  Thus $u \not \in V^-(G-v)$ and by symmetry, $v \not\in V^-(G-u)$. 

(ii) $\Leftarrow$
Let both $u$ and $v$ be $\gamma$-bad vertices of $G$, 
$u \not\in V^-(G-v)$ and $v \not\in V^-(G-u)$. 
Hence $\gamma(G-\{u,v\}) \geq \gamma(G)$. 
Consider any $\gamma$-set $M$ of $G_{u,v,1}$. 
If one of $u$ and $v$ belongs to $M$, then $\gamma(G) + 1 = \gamma(G_{u,v,1})$. 
So, let $x_1$ is in each $\gamma$-set of $G_{u,v,1}$. 
But then $pn[x_1, M] = \{x_1,u,v\}$. 
Hence $\gamma(G_{u,v,1}) -1 = \gamma(G-\{u,v\}) \geq \gamma(G) \geq \gamma(G_{u,v,1}) -1$. 
\end{proof}

\begin{corollary}\label{minus1}
 Let $G$ be a noncomplete  graph. 
 Then    $1 \leq \overline{e}pa (G) \leq \overline{E}pa (G)$. 
 Moreover, (a) $\overline{e}pa (G)  = 1$ if and only if 
there are nonadjacent $\gamma$-bad vertices  $u$ and $v$ of $G$ 	
such that $u \not\in V^-(G-v)$ and $v \not\in V^-(G-u)$,  
and (b) $\overline{E}pa (G)=1$ if and only if $\gamma(G)=1$. 
\end{corollary}

\begin{proof}
Observation \ref{chain2} implies $1 \leq \overline{e}pa (G)$. 

(a) Immediately by Theorem \ref{pa1over}.

(b) 	If $\gamma(G) = 1$ then clearly $\overline{E}pa (G)=1$.
       If $\gamma(G) \geq 2$ then $G$ has $2$ nonadjacent vertices 
       at least one of which is $\gamma$-good. 
			By Theorem \ref{pa1over}, $\overline{E}pa (G) \geq 2$. 		
\end{proof}

\begin{theorem}\label{nepn=2}
Let $u$ and $v$ be nonadjacent vertices of a graph $G$. 
Then  $\gamma(G) \leq \gamma (G_{u,v,2}) \leq \gamma(G) +1$. 
Moreover, 
\begin{itemize}
\item[($\mathbb{C}$)]  $\gamma (G_{u,v,2}) = \gamma(G)$ if and only if 
                                              one of the following holds:
\begin{itemize}
\item[(i)] there is a $\gamma$-set of $G$ which contains both $u$ and $v$. 
\item[(ii)] at least one of $u$ and $v$ is in $V^-(G)$.  
\end{itemize}
\item[($\mathbb{D}$)]   $\gamma (G_{u,v,2}) = \gamma(G)+1$
                                                if and only if  $u,v \not\in V^-(G)$ and any $\gamma$-set of $G$ 
																							contains at most one of $u$ and $v$. 
\end{itemize}
\end{theorem}
\begin{proof}
For any $\gamma$-set $D$ of $G$, $D\cup \{x_2\}$ is a dominating set of $G_{u,v,2}$. 
 Hence $\gamma(G_{u,v,2}) \leq \gamma(G) +1$. 
Suppose $\gamma(G_{u,v,2}) \leq \gamma(G)-1$  and let $M$ be a $\gamma$-set of $G_{u,v,2}$.
 Then at least one of $x_1$  and $x_2$ is in $M$.
 If $x_1, x_2 \in M$ then $M_1 = (M-\{x_1,x_2\}) \cup \{u,v\}$ 
is a dominating set of $G$ and $|M_1| \leq \gamma(G_{u,v,2})$, a contradiction. 
So let without loss of generality, $x_1 \in M$ and $x_2 \not\in M$. 
If $u  \in M$ or $v \in M$ then again $M_1$ is a dominating set of $G$
 and $|M_1| \leq \gamma(G_{u,v,2})$, a contradiction. 
Thus $x_1 \in M$  and $u,v \not\in M$. But then $(M-\{x_1\}) \cup \{u\}$ is 
a dominating set of $G$,  contradicting $\gamma(G_{u,v,2}) < \gamma(G)$. 
Thus $\gamma(G) \leq \gamma(G_{u,v,2}) \leq  \gamma(G) + 1$.

($\mathbb{C}$) $\Rightarrow$
      Let $\gamma (G_{u,v,2}) = \gamma(G)$. 
			Assume that neither (i) nor (ii) hold. 
			Let $M$ be a $\gamma$-set of $G_{u,v,2}$. 
			If $x_1,x_2 \in M$ then $M_1 = (M-\{x_1,x_2\}) \cup \{u,v\}$ 
			is a dominating set of $G$ of cardinality not more than $\gamma(G)$ 
			and $u,v \in M_1$, a contradiction. 
			Let without loss of generality $x_1 \in M$ and $x_2 \not \in M$. 
			Since $M-\{x_1\}$ is no dominating set of $G$, 
			$u \in pn[x_1,M]$. But then $M_3 = (M-\{x_1\}) \cup \{u\}$ 
			is a $\gamma$-set of $G$ and $u \in V^-(G)$, a contradiction. 
      Thus at least one of (i) and (ii) is valid. 

($\mathbb{C}$) $\Leftarrow$
     If both $u$ and $v$ belong to some $\gamma$-set $D$ of $G$, then 
      $D$ is a dominating set of $G_{u,v,2}$. Hence $\gamma(G_{u,v,2}) = \gamma(G)$. 
     Finally let $u \in V^-(G)$ and $D$ a $\gamma$-set of $G-u$. 
		Then $D \cup \{x_1\}$ is a dominating set of $G_{u,v,2}$ of 
		cardinality $\gamma(G)$. Thus $\gamma(G_{u,v,2}) = \gamma(G)$. 
		
($\mathbb{D}$)		Immediately by ($\mathbb{C}$) and $\gamma(G) \leq \gamma (G_{u,v,2}) \leq \gamma(G) +1$. 
\end{proof}

\begin{corollary} \label{minus11}
Let $G$ be a noncomplete  graph. Then
(a)  $\overline{e}pa(G) \leq  2$ if and only if 
        there  are nonadjacent vertices $u,v \in V(G) - V^-(G)$ 
				such that  any $\gamma$-set of $G$ contains at most one of them, and 
(b)  $\overline{E}pa(G) =  2$   if and only if $\gamma(G) \geq 2$ and 
               each $\gamma$-set of $G$ is a clique. 
	\end{corollary}
\begin{proof}
(a) Immediately by Theorem \ref{nepn=2}.

(b) $\Rightarrow$ Let $\overline{E}pa(G) =  2$. 
      By Corollary \ref{minus1}, $\gamma(G) \geq 2$. 
			Suppose $G$ has a $\gamma$-set, say $D$, which is not a clique. 
			Then there are nonadjacent $u,v \in D$. By Theorem \ref{nepn=2} ($\mathbb{C}$), 
			$\gamma(G_{u,v,2}) = \gamma(G)$, which contradict $\overline{E}pa(G) =  2$. 
			Thus, each $\gamma$-set of $G$ is a clique. 
			
	(b) $\Leftarrow$  Let $\gamma(G) \geq 2$ and let each $\gamma$-set of $G$ be  
			   a clique. 
						If $G$ has a vertex $z \in V^-(G)$ and $M_z$ is a $\gamma$-set of $G-z$, 
				then $M = M_z \cup \{z\}$ is a $\gamma$-set of $G$ and $z$ is an isolated vertex 
				of the graph induced by $M$, a contradiction. 
				Thus $V^-(G)$ is empty. 
				Now by Theorem \ref{nepn=2} ($\mathbb{D}$), $\overline{E}pa(G) =  2$.				
\end{proof}

\begin{example}\label{KMN}
The join of two graphs $G_1$ and $G_2$ with disjoint vertex sets 
 is the graph, denoted by $G_1 + G_2$, with the vertex set $V(G_1) \cup V(G_2)$ 
 and edge set $E(G_1) \cup E(G_2) \cup  \{uv  \mid u \in V(G_1), v \in V(G_2)\}$. 
Let $\gamma(G_i) \geq 3$, $i=1,2$. Then $\gamma(G_1 + G_2) = 2$ 
and each $\gamma$-set of $G_1 + G_2$ contains exactly one vertex of $G_i$, $i=1,2$. 
Hence $\overline{E}pa(G_1 + G_2) =  2$.
In particular, $\overline{E}pa(K_{m,n}) =  2$  when $m,n \geq 3$.
\end{example}

\begin{theorem}\label{otri}
Let $u$ and $v$ be nonadjacent vertices of a graph $G$.  
Then  $\gamma(G) \leq \gamma (G_{u,v,3}) \leq \gamma(G) +1$. 
Moreover,   $\gamma (G_{u,v,3}) = \gamma(G)$ if and only if 
                              at least one of the following holds:
\begin{itemize}
\item[(i)] $u \in V^-(G)$ and   $v$ is a $\gamma$-good vertex of $G-u$, 
\item[(ii)] $v \in V^-(G)$ and $u$ is a $\gamma$-good vertex of $G-v$.
\end{itemize}			
	\end{theorem}
\begin{proof}
If $D$ is a dominating set of $G$, then $D \cup \{x_2\}$ is a dominating set of $G_{u,v,3}$. 
Hence $\gamma(G_{u,v,3}) \leq \gamma(G)+1$. 
We already know that $\gamma(G) \leq \gamma(G_{u,v,2})$ and 
$\gamma(G_{u,v,2}) \leq \gamma(G_{u,v,3})$. But then $\gamma(G) \leq \gamma(G_{u,v,3})$. 

 $\Rightarrow$    Let $\gamma (G_{u,v,3}) = \gamma(G)$ and let $M$ be 
                                   a $\gamma$-set of $G_{u,v,3}$ such that $Q = M \cap \{x_1,x_2,x_3\}$ 
																		has minimum cardinality. Clearly $|Q| = 1$.  
																		If $\{x_2\} = Q$ then $M-\{x_2\}$ is a dominating set of $G$, 
																		contradicting $\gamma (G_{u,v,3}) = \gamma(G)$. 
																		Let without loss of generality $\{x_1\} = Q$.  
																		This implies $v \in M$, $x_3 \in pn[v, M]$ and $pn[x_1,M] = \{u,x_1,x_2\}$. 
																		Then $M_2 = (M -\{x_1\}) \cup \{u\}$ is a $\gamma$-set of $G$, 
																		$pn[u,M_2] = \{u\}$ and $v \in M_2$; hence (i) holds.

$\Leftarrow$  Let without loss of generality (i) is true. 
                              Then there is a $\gamma$-set $D$ of $G$ such that $u,v \in D$ and $D-\{u\}$ 
															is a $\gamma$-set of $G-u$. But then $(D-\{u\}) \cup \{x_1\}$ is a dominating set 
															of $G_{u,v,3}$, which implies $\gamma(G) \geq \gamma(G_{u,v,3})$.
\end{proof}

\begin{corollary} \label{oe3}
Let $G$ be a noncomplete  graph. Then 
\begin{itemize}
\item[(i)] $\overline{e}pa(G) \leq 3$ if and only if  there is a pair of nonadjacent vertices
                   $u$ and $v$ such that neither  (i) nor (ii) of Theorem \ref{otri} is valid.
\item[(ii)] $\overline{e}pa(G) = \overline{E}pa(G) = 3$   if and only if 
                    all vertices of $G$ are $\gamma$-good, $V^-(G)$ is empty and 
                    for every $2$  nonadjacent vertices $u$ and $v$ of $G$ 
		                 there is a $\gamma$-set  of $G$  which contains them both.  									
		\end{itemize} 			
\end{corollary}
\begin{proof}
(ii) $\Rightarrow$: Let $\overline{e}pa(G) = \overline{E}pa(G) = 3$. 
       If $u \in V^-(G)$ and $D$ is a $\gamma$-set of $G-u$, 
			then for $u$ and each $v \in D$ is fulfilled (i) of Theorem \ref{otri}. 
			But then $\overline{E}pa(G) \not= 3$, a contradiction. So, 
			$V^-(G)$ is empty. 
			Suppose that $G$ has $\gamma$-bad vertices. 
			Then there is a $\gamma$-bad vertex which is nonadjacent 
			to some other vertex of $G$.  But Theorem \ref{nepn=2}($\mathbb{D}$) 
			implies $\overline{e}pa(G)  < 3$, a contradiction. 
			Thus all vertices of $G$ are $\gamma$-good. 
			Now let $u,v \in V(G)$ be nonadjacent. 
			If there is no $\gamma$-set of $G$ which contains both $u$ and $v$, 
			then by Theorem \ref{nepn=2}($\mathbb{D}$) we have 
			$\gamma(G_{u,v,2}) = \gamma(G) +1$, a contradiction.
			
(ii) $\Leftarrow$: Let $V^-(G)$ be empty and for each pair $u,v$ of 
       nonadjacent vertices of $G$ there is a $\gamma$-set $D_{uv}$ of $G$ with $u,v \in D_{uv}$. 
			By Theorem \ref{otri}, $\gamma(G_{u,v,3}) = \gamma(G) +1$, and 
			by Theorem \ref{nepn=2}, $\gamma(G_{u,v,2}) = \gamma(G)$. 
			Hence $pa(u,v) =3$.  
\end{proof}

\begin{example}\label{opa3}
Denote by $\mathcal{U}$ the class of all graphs $G$  with $\overline{e}pa(G) = \overline{E}pa(G) = 3$. 
Then all the following holds. 
 (a) $C(2k+1; \{\pm 1,\pm 2,...,\pm (k-1)\}) \in \mathcal{U}$ for all $k \geq 1$. 
(b) Let $G$ be a nonconnected graph. Then $G \in \mathcal{U}$ if and only if 
      $G$ has no isolated vertices and each its component is either in $\mathcal{U}$ or is complete.
\end{example}

\begin{theorem}\label{four}
Let $u$ and $v$ be nonadjacent vertices of a graph $G$. 
Then  $\gamma(G) \leq \gamma (G_{u,v,4}) \leq \gamma(G) +2$. 
Moreover, the following assertions are valid.
\begin{itemize}
\item[($\mathbb{E}$)] $\gamma (G_{u,v,4}) = \gamma(G) +2$ if and only if $\gamma (G_{u,v,1})  = \gamma(G) +1$, 
\item[($\mathbb{F}$)] If $\gamma (G_{u,v,1})  = \gamma(G)$ and $\gamma (G_{u,v,i})  = \gamma(G) +1$ 
                                              for some $i \in \{2,3\}$, then $\gamma (G_{u,v,4})  = \gamma(G) +1$. 
\item[($\mathbb{G}$)] Let $\gamma (G_{u,v,3}) = \gamma(G)$.
                                             Then  $\gamma (G_{u,v,4}) \leq \gamma(G) +1$ and 
                                             the equality holds if and only if $\gamma(G-\{u,v\}) \geq \gamma(G)-1$.
\item[($\mathbb{H}$)]  	$\gamma (G_{u,v,4}) = \gamma(G)$	 if and only if 		$\gamma(G-\{u,v\}) = \gamma(G)-2$.																		
 \end{itemize} 
\end{theorem}

\begin{proof}
Since $\gamma(G) \leq \gamma (G_{u,v,3})$ (by Theorem \ref{otri}) and  
$\gamma (G_{u,v,3}) \leq \gamma (G_{u,v,4})$ (by Observation \ref{chain2}), 
 we have $\gamma(G) \leq \gamma (G_{u,v,4})$. 
Let $S$ be a $\gamma$-set of $G$. Then $S \cup \{x_1,x_4\}$ is a dominating set 
of $G_{u,v,4}$, which leads to $\gamma(G_{u,v,4}) \leq \gamma(G)+2$. 
\medskip

{\bf Claim 1.} \ If $\gamma(G_{u,v,1}) \leq \gamma(G)$ then 
                           $\gamma(G_{u,v,4}) \leq \gamma(G) + 1$. 
\begin{proof}[Proof of Claim 1] 													
Assume that $v$ is a $\gamma$-bad vertex of $G$, $u \in V^-(G-v)$ and $R$ a 
$\gamma$-set of $G-\{u,v\}$. Then $|R| = \gamma((G-v)-u) = \gamma(G-v)-1 
= \gamma(G)-1$ and $R \cup \{x_1,x_4\}$ is a dominating set of $G_{u,v,4}$. 
Hence $\gamma(G_{u,v,4}) \leq |R| +2 = \gamma(G)+1$. 

Assume now that $D$ is a $\gamma$-set of $G$ with $u \in D$. 
Then $D \cup \{x_3\}$ is a dominating set of $G_{u,v,4}$. Hence again $\gamma(G_{u,v,4}) 
\leq \gamma(G) +1$.  Now by Theorem \ref{pa1over} we immediately obtain the required. 
  \end{proof}

($\mathbb{E}$) 
                               Let  $\gamma (G_{u,v,4}) = \gamma(G) + 2$.
                               By Claim 1,  $\gamma (G_{u,v,1}) > \gamma(G) $ and by 
															Theorem \ref{pa1over},  $\gamma (G_{u,v,1}) = \gamma(G) + 1$. 
															
															Let now   $\gamma (G_{u,v,1})  =  \gamma(G) +1$. 
															By Theorem \ref{pa1over} $u$ and $v$ are $\gamma$-bad vertices of $G$, 
															$u \not\in V^-(G-v)$ and $v \not\in V^-(G-u)$.  Let $M$ be a $\gamma$-set of 
															$G_{u,v,4}$ such that $R = M \cap \{x_1,x_2,x_3,x_4\}$ has  minimum 
															cardinality. Clearly $|R| \in \{1,2\}$. 
															Assume first $|R| = 1$ and without loss of generality $\{x_2\} =  M$. 
															Then $M-\{x_2\}$ is a dominating set of $G$ with $v \in M-\{x_2\}$. 
															Since $v$ is a $\gamma$-bad vertex of $G$, $|M-\{x_2\}| > \gamma(G)$ 
															and then $\gamma(G_{u,v,4}) = |M| > \gamma(G)+1$. 
															Let now $|R| = 2$ and without loss of generality $x_1,x_4 \in M$. 
															Since  $|M \cap \{x_1,x_2,x_3,x_4\}|$ is minimum, 
															$u,v \not\in M$ and $M-\{x_1,x_4\}$ is a dominating set of $G-\{u,v\}$. 
															But then  $\gamma (G_{u,v,4}) = 2 + |M-\{x_1,x_4\}| \geq 2 + \gamma((G-u)-v) 
															\geq 2 + \gamma(G-u) = 2+\gamma(G)$. 															
\medskip

($\mathbb{F}$)     Let  $\gamma (G_{u,v,1}) = \gamma(G)$.
                              By Claim 1, $\gamma (G_{u,v,4}) \leq \gamma(G) +1$. 
															If $\gamma (G_{u,v,i}) = \gamma(G) +1$ for some $i \in \{1,2\}$, 
															then since $\gamma (G_{u,v,4}) \geq \gamma (G_{u,v,i})$, 
															we obtain $\gamma (G_{u,v,4}) = \gamma(G) +1$. 
\medskip

 ($\mathbb{G}$)  Let $\gamma (G_{u,v,3}) = \gamma(G) $. 
                                Hence at least one of (i) and (ii) of Theorem \ref{otri} holds, 
                                 and by    ($\mathbb{E}$), $\gamma (G_{u,v,4}) \leq \gamma(G) +1$.
																
											Assume that the equality holds. 
If $\gamma(G-\{u,v\}) = \gamma(G)-2$ then for any 
$\gamma$-set $U$ of $G- \{u,v\}$,  
	$U \cup \{x_1,x_4\}$ is a dominating set of $G_{u,v,4}$. 
	Hence $\gamma(G_{u,v,4}) = \gamma(G)$, a contradiction. 						

Let  now $\gamma(G-\{u,v\}) \geq \gamma(G)-1$ 
and without loss of generality condition (i) of Theorem \ref{otri} is satisfied. 
 	Suppose $\gamma(G_{u,v,4}) = \gamma(G)$.  
	Hence for each $\gamma$-set $M$ of $G_{u,v,4}$ are fulfilled: $x_1,x_4 \in M$, 
	$x_2,x_3,u,v \not\in M$, $pn[x_1,M] = \{x_1,x_2,u\}$ and $pn[x_4,M] = \{x_3,x_4,,v\}$. 
	But then $\gamma(G-\{u,v\}) = \gamma(G)-2$, a contradiction. 	
 Thus $\gamma (G_{u,v,4}) = \gamma(G) +1$. 
\medskip

 ($\mathbb{H}$)
If $\gamma (G_{u,v,4}) = \gamma(G)$ then $\gamma (G_{u,v,3}) = \gamma(G)$
and by  ($\mathbb{G}$), $\gamma(G-\{u,v\}) = \gamma(G)-2$.

Now let $\gamma(G-\{u,v\}) = \gamma(G)-2$.
But then for each $\gamma$-set $D$ of $G-\{u,v\}$, 
the set $D \cup \{x_1,x_4\}$ is a dominating set of $G_{u,v,4}$. 
Thus $\gamma (G_{u,v,4}) = \gamma(G)$.
	\end{proof}

\begin{theorem}\label{five}
Let $u$ and $v$ be nonadjacent vertices of a graph $G$. 
If $\gamma (G_{u,v,k}) = \gamma(G) $ then $k \leq 4$. 
If $k \geq 5$ then  $\gamma (G_{u,v,k}) > \gamma(G) $. 
If $\gamma (G_{u,v,4}) = \gamma(G) $ then $\gamma (G_{u,v,5}) = \gamma(G) + 1$.
\end{theorem}

\begin{proof}
By Theorem \ref{four}, $\gamma(G) \leq \gamma (G_{u,v,4}) \leq \gamma(G) +2$. 
If $\gamma (G_{u,v,4}) > \gamma(G)$ then $\gamma (G_{u,v,k}) > \gamma(G)$ 
for all $k \geq 5$ because of Observation \ref{chain2}. 
So, let $\gamma (G_{u,v,4}) =  \gamma(G)$.  
By Theorem \ref{four}($\mathbb{H}$), $\gamma(G-\{u,v\}) = \gamma(G)-2$. 
But then for each $\gamma$-set $D$ of $G-\{u,v\}$, the set 
$D \cup \{x_1,x_3,x_5\}$ is a dominating set of $G_{u,v,5}$. 
Hence $\gamma (G_{u,v,5}) \leq \gamma(G) + 1$. 
Let now $M$ be a $\gamma$-set of $G_{u,v,5}$. 
Then at least one of $x_2, x_3, x_4$ is in $M$ and 
hence $ \gamma (G_{u,v,5}) = |M| \geq \gamma(G) + 1$. 
Thus $\gamma (G_{u,v,5}) = \gamma(G) + 1$.
Now  using again  Observation \ref{chain2} we conclude
 that $\gamma (G_{u,v,k}) > \gamma(G)$ for all $k \geq 5$.
\end{proof}

\begin{corollary} \label{oe45}
Let $G$ be a noncomplete  graph. Then $\overline{e}pa(G) \leq \overline{E}pa(G) \leq 5$. 
Moreover
\begin{itemize}
\item[(i)] $\overline{E}pa(G) = 5$ if and only if  there are 
									nonadjacent vertices $u$ and $v$ of $G$ with $\gamma(G-\{u,v\}) = \gamma(G)-2$.
\item[(ii)] $\overline{e}pa(G) = 5$  if and only if $G$ is edgeless. 
\item[(iii)] $\overline{e}pa(G) = \overline{E}pa(G)  = 4$  if and only if 
                     for each pair $u,v$ of nonadjacent vertices of $G$, 
											$\gamma(G-\{u,v\}) \geq \gamma(G)-1$
										   and at least one  of the following holds:
\begin{itemize}
\item[(a)] $u \in V^-(G)$ and   $v$ is a $\gamma$-good vertex of $G-u$, 
\item[(b)] $v \in V^-(G)$ and $u$ is a $\gamma$-good vertex of $G-v$.
\end{itemize}		
\end{itemize}
\end{corollary}

\begin{proof}
By Theorem \ref{five}, $\overline{E}pa(G) \leq 5$. 

(i) $\Rightarrow$  Let $\overline{E}pa(G) = 5$. 
     Then there is a pair $u,v$ of nonadjacent vertices of $G$ such that 
		$\gamma(G_{u,v,4}) = \gamma(G)$. 
		Now by Theorem \ref{four}($\mathbb{H}$), $\gamma(G-\{u,v\}) = \gamma(G)-2$. 
				
(i) $\Leftarrow$  Let 		$\gamma(G- \{u,v\}) = \gamma(G) -2$ and $D$ a $\gamma$-set of $G-\{u,v\}$, 
                                  where $u$ and $v$ are nonadjacent vertices of $G$. 
                                  Hence $D_1 = D \cup \{x_1,x_4\}$    is a dominating set of $G_{u,v,4}$ and 
			                            $|D_1| = \gamma(G)$. This implies $\gamma (G_{u,v,4}) = \gamma(G)$. 
																	The result now follows by Theorem \ref{five}.
			
(ii)                              If $G$ has no edges, then the result is obvious. 
                                   So let $G$ have edges and  $\overline{e}pa(G) = 5$. 
																	Then for any $2$ nonadjacent vertices $u$ and $v$ of $G$
																	is satisfied $\gamma(G- \{u,v\}) = \gamma(G) -2$  (by (i)). 
															    Hence we can choose $u$ and $v$ so that they have a neighbor in common, say $w$.  
																	But then $w$ is a $\gamma$-bad vertex of $G-u$ which implies $v \not\in V^-(G-u)$. 
																	This leads to $\gamma(G- \{u,v\}) \geq \gamma(G) -1$, a contradiction.

(iii) $\Rightarrow$  Let  $\overline{e}pa(G) = \overline{E}pa(G) =4$. 
                                     Then for each two nonadjacent   $u,v \in V(G)$ we have  
																			$\gamma(G) = \gamma(G_{u,v,3}) < \gamma(G_{u,v,4})$. 
																			Now by Theorem \ref{four}($\mathbb{G}$), 
																			$\gamma(G-\{u,v\}) \geq \gamma(G)-1$ and 
																			by Theorem \ref{otri}, at least one of (a) and (b) is valid.

(iii) $\Leftarrow$  Consider any two  nonadjacent vertices $u,v$ of $G$. 
                                  Then   $\gamma(G-\{u,v\}) \geq \gamma(G)-1$
										                 and  at least one  of (a) and (b) is valid. 
																	 Theorem \ref{otri}  now implies $\gamma(G) = \gamma(G_{u,v,3})$, and 
																		by Theorem \ref{four}, $pa(u,v) = 4$. 
		\end{proof}

\begin{example}\label{knn}
Let $G_n$ be the Cartesian product of two copies of  $K_n$, $n \geq 2$. 
We consider $G_n$ as an $n \times n$ array of vertices $\{x_{i,j} \mid 1\leq i \leq j \leq n \}$, 
where the closed neighborhood of $x_{i,j}$ is the union of the sets 
$\{x_{1,j}, x_{2,j},...,x_{n,j}\}$ and $\{x_{i,1}, x_{i,2},...,x_{i,n}\}$. 
Note that $V(G_n) = V^-(G_n)$ and $\gamma(G_n) = n$\cite{hr}.
 It is easy to see that the following sets are $\gamma$-sets of $G_n-x_{1,1}$: 
$D_i = \{x_{2,i},x_{3,i+1},...,x_{n,n+i-2}\}$, $i=2,3,...,n$, 
where $x_{k,j} := x_{k,j-n+1}$ for $j>n$ and $2 \leq k \leq n$. 
Since $D = \cup_{i=2}^nD_i = V(G_n)-N[x_{1,1}]$, 
all $\gamma$-bad vertices of $G_n-x_{1,1}$ are the neighbors of $x_{1,1}$ in $G_n$. 
Since each vertex of $D$ is adjacent to some neighbor of $x_{1,1}$, 
$V^-(G_n-x_{1,1})$ is empty. 
Now by Theorem \ref{otri} we have  $pa(x_{1,1}, y)  \geq 4$, 
and by Theorem \ref{four}($\mathbb{H}$), $pa(x_{1,1}, y)  < 5$. 
Thus $pa(x_{1,1}, y)  = 4$. 
By reason of symmetry, we obtain $\overline{e}pa(G_n) = \overline{E}pa(G_n) =4$.
\end{example}

\section{Observations and open problems}

A  constructive characterization of the trees $T$ with $i(T) \equiv \gamma(T)$, 
and therefore a constructive characterization of the trees $T$ with $Epa(T) =2$ (by Corollary \ref{22}), 
was provided in \cite{hhs}.

\begin{problem}\label{unic}
 Characterize all unicyclic graphs $G$ with  $Epa(G) =2$.
\end{problem}

\begin{problem}\label{ov2}
Find results on $\gamma$-excellent graphs $G$ with $\overline{E}pa(G) =  2$.
\end{problem}

\begin{problem}\label{ee4}
 Characterize all  graphs $G$ with  $\overline{e}pa(G) = \overline{E}pa(G) = 4$.
\end{problem}

\begin{corollary}\label{sum}
Let $G$ be a connected noncomplete graph with edges. Then: 
\begin{itemize}
\item[(i)] $2 \leq epa(G) + \overline{E}pa(G) \leq 8$,
\item[(ii)] $2 \leq epa(G) + \overline{e}pa(G) \leq 7$, 
\item[(iii)] $3 \leq Epa(G) + \overline{E}pa(G) \leq 8$, 
\item[(iv)] $3 \leq Epa(G) + \overline{e}pa(G) \leq 7$. 
\end{itemize}
\end{corollary}

\begin{proof}
(i)--(iv): The left-side inequalities  immediately follow by 
              Corollary \ref{1}  and  Corollary \ref{minus1}. 
							The right-side inequalities  hold  because of  
              Corollary \ref{epn3}  and  Corollary \ref{oe45}. 
\end{proof}

Note that all bounds stated in Corollary \ref{sum} are attainable. 
We leave finding examples demonstrating this to the reader.

\begin{problem}\label{seven}
 Characterize all  graphs $G$ for which    $epa(G) + \overline{e}pa(G) = 7$.
\end{problem}
All $2k$-order $(2k-2)$-regular graphs, $k \geq 2$, are such graphs. 

\begin{problem}\label{new}
 Define appropriate the numbers $epa_{\mu}(G)$, $Epa_{\mu}(G)$, 
$\overline{e}pa_{\mu}(G)$ and $\overline{E}pa_{\mu}(G)$, where $\mu$ 
 is at least one of the independent/connected/total/restrained/ double/acyclic/signed/minus/Roman 
 domination number.   Find results on these numbers. 
\end{problem}

\end{document}